\documentclass[11pt]{article}

\usepackage[sort]{natbib}
\setcitestyle{numbers,square,comma}
\usepackage{geometry}
\usepackage{amsthm, amsmath, amssymb}
\usepackage{graphicx, color}
\usepackage[font=small,labelfont=bf]{caption} 
\usepackage{float}
\usepackage{makecell}
\usepackage[dvipsnames]{xcolor}
\usepackage{mathtools}
\usepackage[hyphens]{url}
\usepackage{hyperref}
\usepackage[hyphenbreaks]{breakurl}
\hypersetup{colorlinks=true,
			urlcolor=black,
			linkcolor=black,
			citecolor=blue,
			bookmarksdepth=paragraph,
			pdftitle={Convex computation of maximal Lyapunov exponents (arXiv version)},
			pdfauthor={Hans Oeri, David Goluskin}}

\newtheorem{prop}{Proposition}
\newtheorem{remark}{Remark}

\usepackage[noabbrev]{cleveref}
\crefname{equation}{}{}
\crefname{prop}{proposition}{propositions}
\crefname{remark}{remark}{remarks}

\newcommand{\mL}{m(\Lambda)} % Haar probability measure
\newcommand{\norm}[1]{\vert #1 \vert }

\newcommand{\T}{{\sf T}}
\newcommand{\sfA}{{\sf A}}
\newcommand{\sfM}{{\sf M}}
\newcommand{\Df}{D\hspace{-1pt}f}
\newcommand{\R}{\mathbb{R}}
\newcommand{\cB}{\mathcal{B}}
\newcommand{\cC}{\mathcal{C}}
\newcommand{\cV}{\mathcal{V}}
\newcommand{\Rm}{\mathbb{R}^n} % A shorthand for R^n
\newcommand{\Sn}{\mathbb{S}^{n-1}}
\newcommand{\Rx}{\mathbb{R}[x]}% Shorthand for space of infinite degree, n-variate polynomials,
\newcommand{\Rxz}{\mathbb{R}[x,z]} % Shorthand for space of infinite degree, 2n-variate polynomials
\newcommand{\B}{\mathcal{B}} % The set compact absorbing set 
\newcommand{\U}{B} % The bound B of the optimization problems
\newcommand{\dt}{\mathrm{d}t}
\newcommand{\ddt}{\frac{\mathrm{d}}{\dt}}
\newcommand{\ov}{\overline}
\newcommand{\beq}{\begin{equation}}
\newcommand{\eeq}{\end{equation}}
\newcommand{\cG}{\mathcal{G}} % Macros for a group

\newcommand\solidrule[1][15pt]{\rule[0.5ex]{#1}{1pt}}
\newcommand\dashedrule{\mbox{%
		\solidrule[3pt]\hspace{3pt}\solidrule[3pt]\hspace{3pt}\solidrule[3pt]}}
\newcommand\shortdottedrule{\mbox{%
		$\cdot$\hspace{1.5pt}$\cdot$\hspace{1.5pt}$\cdot$\hspace{1.5pt}$\cdot$\hspace{1.5pt}$\cdot$}}

\title{Convex computation of maximal Lyapunov exponents}

\author{Hans Oeri, David Goluskin\thanks{Email: {\tt goluskin@uvic.ca}}}
\date{Department of Mathematics and Statistics, University of Victoria, Victoria, Canada}

\begin{document}

\maketitle

\begin{abstract}
We describe an approach for finding upper bounds on an ODE dynamical system's maximal Lyapunov exponent among all trajectories in a specified set. A minimization problem is formulated whose infimum is equal to the maximal Lyapunov exponent, provided that trajectories of interest remain in a compact set. The minimization is over auxiliary functions that are defined on the state space and subject to a pointwise inequality. In the polynomial case---i.e., when the ODE's right-hand side is polynomial, the set of interest can be specified by polynomial inequalities or equalities, and auxiliary functions are sought among polynomials---the minimization can be relaxed into a computationally tractable polynomial optimization problem subject to sum-of-squares constraints. Enlarging the spaces of polynomials over which auxiliary functions are sought yields optimization problems of increasing computational cost whose infima converge from above to the maximal Lyapunov exponent, at least when the set of interest is compact. For illustration, we carry out such polynomial optimization computations for two chaotic examples: the Lorenz system and the H\'enon--Heiles system. The computed upper bounds converge as polynomial degrees are raised, and in each example we obtain a bound that is sharp to at least five digits. This sharpness is confirmed by finding trajectories whose leading Lyapunov exponents approximately equal the upper bounds.
\end{abstract}

\section{Introduction}
\label{sec: intro}

In a dynamical system, the rates at which infinitesimally nearby trajectories converge or diverge are captured by Lyapunov exponents (LEs)~\cite{Pikovsky2016}. These diagnostic exponents are especially prominent in the study of chaotic systems, where a positive exponent corresponds to exponentially diverging trajectories and, therefore, to sensitive dependence on initial conditions. For trajectories that are periodic orbits, LEs coincide with Floquet multipliers. If periodic orbits are linearly unstable and embedded in a chaotic attractor, their positive multipliers are one measure of their relative importance to the chaotic dynamics~\cite{Cvitanovic2016}. However, the instability of trajectories that positive LEs indicate also makes it difficult to compute such LEs precisely using numerical integration~\cite{Dieci1997, Eckmann1985, Geist1990}. The present work is concerned with LEs that are maximal among all trajectories in a chosen set. Rather than using numerical integration, we formulate convex optimization problems that yield convergent upper bounds on maximal LEs.

Consider dynamics of $x(t)$ in $\R^n$ governed by an autonomous ODE system,
\beq
\label{eq: dynamical system}
\ddt x(t) = f(x(t)) , \qquad x(0) = x_0.
\eeq
We ensure the system is well-posed by assuming $f\in\cC^1(\Rm,\Rm)$, where $\cC^k(\Rm,\Rm)$ denotes the space of $k$-times continuously differentiable functions mapping $\Rm$ to $\Rm$, and similarly $\cC^k(\Rm)$ denotes functions mapping $\Rm$ to $\R$. At each point along $x(t)$, the linearization of~\cref{eq: dynamical system} defines dynamics of $y(t)$ evolving in the tangent space $\R^n$ according to
\beq
\label{eq: linearized system}
\ddt y(t)=\Df(x(t))\,y(t), \qquad y(0)=y_0,
\eeq
where $\Df(x)$ is the $n\times n$ Jacobian matrix for $f(x)$. The perturbation vector from a trajectory $x(t)$ to a very nearby trajectory evolves proportionally to $y(t)$, so long as the trajectories remain close enough for the linearization \cref{eq: linearized system} to be accurate. If two trajectories separate on average and evolve nonlinearly, eventually \cref{eq: linearized system} will not describe the vector between them. Thus~\cref{eq: linearized system} may not approximate nearby trajectories uniformly in time, but it captures, at each instant~$t$, the dynamics of trajectories that are infinitesimally near~$x(t)$. Since \cref{eq: linearized system} is linear in $y$, the dynamics of $y$ are independent of its magnitude. Thus we fix $|y_0|=1$, so $y(t)$ initially lies on the sphere $\Sn\subset \Rm$ but can grow or shrink in time. 

For each initial condition $x_0\in \Rm$ and initial direction $y_0\in\Sn$ in the tangent space, define the resulting LE as
\beq
\label{eq: LE def}
\mu(x_0,y_0) = \limsup_{t \rightarrow \infty} \frac{1}{t} \log \norm{y(t)}.
\eeq
Vectors $y(t)$ evolving in the tangent space of the trajectory $x(t)$ shrink or grow like $e^{\mu t}$, on average. Under our assumption that $f\in\cC^1(\R^n,\R^n)$, if $x(t)$ remains bounded then $y(t)$ cannot shrink or grow faster than exponentially, so $\mu(x_0,y_0)$ cannot be $\pm \infty$.

For the trajectory emanating from each $x_0$, the possible values of the limit~\cref{eq: LE def} among all initial directions $y_0$ define an associated spectrum of LEs, $\mu_1(x_0)\ge\mu_2(x_0)\ge\cdots$. There are at most $n$ distinct LEs in the spectrum~\cite{Meiss2017}. The Oseledets ergodic theorem~\cite{Oseledets1968, Pikovsky2016} implies that the leading exponent $\mu_1(x_0)$ is attained for almost every initial direction $y_0$, so often the leading LE is of primary interest, and many authors refer to it as ``the'' LE.

Here we pursue upper bounds that apply to all trajectories starting within any chosen set $\Omega\subset \Rm$. In other words, these upper bounds apply to the \emph{maximal Lyapunov exponent} $\mu^*_\Omega$ defined by
\beq
\label{maximal LE}
\mu_{\Omega}^* = \sup_{\substack{x_0 \in \Omega~~\\~\;y_0 \in \mathbb{S}^{n-1}}} \mu(x_0,y_0).
\eeq
We assume all $x_0\in\Omega$ lead to bounded trajectories; if not, the supremum in \cref{maximal LE} can be restricted to such $x_0$. A bounded trajectory on which a positive $\mu^*_\Omega$ value is attained is, by definition, the most unstable trajectory starting in $\Omega$. The methods we describe for upper-bounding the supremum in \cref{maximal LE} could analogously lower-bound the infimum over $(x_0,y_0)$, meaning they would give lower bounds on the smallest LE $\mu_n$ among trajectories starting in $\Omega$, but here we speak only in terms of upper bounds on the maximal LE.

The difficulty in computing $\mu^*_\Omega$ by numerical integration is twofold. First, for a single $x_0$ leading to positive $\mu_1$, the value of $\mu_1$ is hard to compute accurately because the exponential separation of trajectories magnifies numerical errors as time goes forward. In chaotic systems, the time average $\tfrac1t\log|y(t)|$ often converges extremely slowly with increasing $t$, if at all~\cite{Parker2012, Benettin1980, Hubertus1997, Tancredi2001, Meiss2017}. Second, even if $\mu_1$ can be computed accurately for single trajectories, there may be an infinite number of trajectories with different $\mu_1$ values. This is expected of the unstable periodic orbits embedded in a chaotic attractor, for instance. Searching over such trajectories for the largest $\mu_1$ value is not a convex search, so even if one finds the trajectory that attains the maximal LE, the existence of another trajectory with larger $\mu_1$ cannot be ruled out using numerical integration. These difficulties motivate our current approach, which may complement the search for maximal LEs by numerical integration but is fundamentally different.

Often one is most interested in $\mu_1$ of a chaotic attractor, which will be the $\mu_1$ value of almost every trajectory in its basin of attraction. However, this $\mu_1$ value is usually strictly smaller than the maximal LE $\mu^*_\Omega$ for any $\Omega$ containing the attractor, with the maximal LE being attained only on a particular unstable orbit. The upper bounds we present here will converge to $\mu^*_\Omega$ from above and so would be strictly larger than $\mu_1$ on the attractor, nonetheless they will be better upper bounds than any existing method can give in many cases.

Our approach to bounding the maximal LE $\mu^*_\Omega$ from above is based on two ingredients that are explained in the next section. The first ingredient is a simple formula that expresses $\mu(x_0,y_0)$ as an infinite-time-averaged integral along trajectories $(x(t),y(t)/|y(t)|)$ in the state space $\Rm\times\Sn$. Various versions of this formula have appeared in the dynamical systems literature, including the Furstenberg--Khasminskii formula for stochastic dynamics~\cite{Arnold1995, Bedrossian2022}. Because $\mu^*_\Omega$ can be expressed as a time average that is maximized among certain trajectories, our second ingredient is a method for computing arbitrarily sharp upper bounds on such time averages. In the important case where the right-hand side of an ODE system is polynomial in the state space variables---and in some generalizations of this case---polynomial optimization recently has been used to compute arbitrarily sharp upper bounds on maximal time averages of various quantities~\cite{Chernyshenko2014, Goluskin2018a, Lakshmi2020, Olson2020, Korda2021, Doering2022a}. This approach yields upper bounds as numerically computed minima of sum-of-squares (SOS) programs---convex optimization problems in which the coefficients of polynomial expressions are constrained by requiring that the polynomials admit representations as sums of squares of other polynomials. A polynomial being SOS is a sufficient condition for pointwise nonnegativity that is stronger in general but also is more tractable to enforce computationally.

Starting concurrently with the advent of SOS programming two decades ago~\cite{Nesterov2000, Parrilo2000, Lasserre2001}, methods have been developed to study dynamical systems by solving SOS programs. These methods all involve constructing scalar functions on a dynamical system's state space and requiring that the functions obey pointwise inequalities. The inequalities, which can be enforced by SOS constraints, imply mathematical statements about the dynamics without the need to compute any trajectories. In addition to the method for bounding time averages that we employ here, SOS computations have been used to verify nonlinear stability~\cite{Parrilo2000, Anderson2015, Hafstein2018, Fuentes2022, Kuntz2016}, show that trajectories do or do not enter specified sets~\cite{Henrion2014, Bramburger2020, Parker2021, Cibulka2022a}, bound expectations in stochastic dynamical systems~\cite{Kuntz2016, Fantuzzi2016}, bound extrema over certain trajectories~\cite{Fantuzzi2020, Miller2021a} or over global attractors~\cite{Goluskin2020a}, and more. We note that all of these methods can be formulated in terms of occupation measures, which are certain measures induced by the dynamics. When an optimization over measures is truncated to a finite number of moments of the measure, the resulting problem is related to an SOS optimization problem by convex duality. For the methods presented below we speak in terms of SOS optimization, but there is a dual description in terms of measures and moments; see~\cite{Henrion2021} for an introduction to occupation measures and the relationship between moments and SOS polynomials.

\Cref{sec: continuous} recalls how SOS computations can bound time averages for ODE systems and then applies this approach to bounding LEs. Section \ref{sec: examples} presents computational results showing the convergence of upper bounds to the maximal LE in two chaotic examples: the Lorenz system and the H\'enon--Heiles system. \Cref{sec: con} offers conclusions and future extensions.

\section{Optimization problems for bounding maximal LEs}
\label{sec: continuous}

\Cref{sec: SOS time avg} explains how maximal time averages among ODE trajectories can be bounded above using convex optimization, including how the optimization problems can be relaxed into SOS programs. \Cref{sec: SOS LE} specializes the approach of \cref{sec: SOS time avg} to the task of bounding a maximal LE from above, yielding computationally useful SOS programs along with theoretical sharpness guarantees. \Cref{sec: sym} explains how symmetries in the dynamical system allow symmetries to be imposed in the SOS programs for bounding maximal LEs, which improves computational cost and numerical accuracy. The computational methods of \cref{sec: SOS LE}, including the symmetry restrictions of \cref{sec: sym}, are implemented to bound LEs in the examples of \cref{sec: examples}.

\subsection{Bounding time averages using convex optimization}
\label{sec: SOS time avg}

Consider any continuous function $\Phi\in\cC(\Rm)$ taking values in $\R$, along with its infinite-time average along a trajectory $x(t)$ of the ODE~\cref{eq: dynamical system},
\begin{align}
\label{time-average}
\overline \Phi(x_0) = \limsup_{T \rightarrow \infty} \frac{1}{T} \int_0^T \Phi(x(t)) \, \dt.
\end{align}
We define time averages with $\limsup$ to ensure they exist. (The bounds obtained would be unchanged if $\ov\Phi$ were defined using $\liminf$.) If the trajectory emanating from $x_0$ is bounded forward in time, then so is $\Phi(x(t))$, and thus $\ov\Phi$ is finite.

The strategy for bounding $\ov\Phi$ without knowing any particular trajectories is to introduce a differentiable \emph{auxiliary function} $V\in\cC^1(\Rm)$. For any such $V$, the continuous function $f\cdot\nabla V$, which also maps $\Rm$ to $\R$, has a vanishing infinite-time average along every bounded trajectory. This is true because $\ddt V(x(t))=f(x(t))\cdot \nabla V(x(t))$ along trajectories; in other words, $f\cdot\nabla V$ is the Lie derivative of $V$ along ODE solutions. Since $V(x(t))$ is uniformly bounded forward in time along each bounded trajectory, $\ov{f\cdot\nabla V}=\limsup_{T\to\infty}\frac{1}{T}[V(x(T))-V(x_0)]=0$. For each differentiable $V$ and each set $\cB\subset\Rm$ in which $x(t)$ remains for $t\ge0$, 
\begin{align}
\label{eq: first UB}
\ov\Phi = \ov{\Phi +f\cdot \nabla V} \leq \sup_{x \in \cB} \left [ \Phi(x) + f(x)\cdot\nabla V(x) \right ].
\end{align}
If nothing is known about trajectories' locations aside from each trajectory being bounded, one can always choose $\cB=\Rm$ in \cref{eq: first UB}. For a poor choice of $V$ the upper bound in \cref{eq: first UB} will be conservative---possibly even infinite if $\cB$ is unbounded---but one can seek $V$ that make the right-hand side as small as possible. This optimization of $V$ is described in general in \cref{eq: time avg bounds}, followed by its SOS relaxation in the polynomial case in \cref{sec: time avg bounds SOS}.

\subsubsection{Optimization over auxiliary functions}
\label{eq: time avg bounds}

The upper bound \cref{eq: first UB} holds for each bounded $x(t)$ and differentiable $V$, so one can maximize the left-hand side over initial conditions $x_0$ and minimize the right-hand side over~$V$. Let $\cB_0$ be a set of initial conditions whose trajectories remain bounded, uniformly in $t\ge0$ but not necessarily uniformly in $x_0$. Suppose all of these trajectories eventually remain in a set $\cB$, which may be unbounded. Then \cref{eq: first UB} can be minimized over $x_0\in\cB_0$ and maximized over $V\in\cC^1(\cB)$ to obtain
\beq
\label{eq: weak duality}
\sup_{x_0 \in \cB_0} \overline \Phi \leq \inf_{V \in \mathcal{C}^1(\cB)} \sup_{x \in \cB} \left[ \Phi(x) + f(x)\cdot\nabla V(x) \right].
\eeq
Here $\cC^1(\cB)$ is defined as the restriction of $\cC^1(\Rm)$ functions to a domain $\cB$ (possibly of lower dimension), which is equivalent to the set of functions mapping $\cB$ to $\R$ that admit a differentiable extension to an open neighborhood of $\cB$. In \cref{eq: weak duality} one can always choose $\cB=\Rm$ but often it is possible to choose smaller $\cB$, and this may decrease the upper bound. If $\cB_0$ lies in the basin of a local attractor, for instance, one can choose any $\cB$ that contains the attractor. If $\cB_0$ is forward invariant, meaning no trajectories leave the set forward in time, then one can choose $\cB=\cB_0$. 

When $\cB$ is forward invariant and also compact, and one chooses $\cB_0=\cB$, then \cref{eq: weak duality} is an equality~\cite{Tobasco2018},
\beq
\label{eq: strong duality}
\sup_{x_0 \in \cB} \overline \Phi = \inf_{V \in \mathcal{C}^1(\cB)} \sup_{x \in \cB} \left[ \Phi(x) + f(x)\cdot\nabla V(x) \right].
\eeq
The equality \cref{eq: strong duality} can be seen as a statement of strong convex duality: the right-hand side can be restated with the inner maximization over invariant measures, and the left-hand side can be stated in the same way with the order of minimization and maximization reversed. For a sketch of the proof of~\cref{eq: strong duality}, see expression (17) in~\cite{Tobasco2018}. The duality result~\cref{eq: strong duality} has analogous precedents in the context of discrete-time dynamics, for which we refer to the discussion and references in~\cite{Bochi2018, Jenkinson2019}. In~\cite{Tobasco2018} it is assumed that $\cB$ is a full-dimensional region, but their proof works more generally for any set $\cB$ that is compact and forward invariant~\cite{TobascoPC}. This generality is needed in what follows, where we apply \cref{eq: strong duality} to sets defined by Cartesian products involving the sphere $\Sn$.

Although the infimum on the right-hand sides of \cref{eq: weak duality,eq: strong duality} may not be attained, there exist $V$ giving arbitrarily sharp upper bounds on the left-hand supremum. However, it can be difficult to evaluate the inner supremum on the right-hand side for a given $V$, and optimizing over $V$ is still harder. An important exception, in which computationally tractable relaxations of the right-hand minimax problem are possible, is the case where the ODE right-hand side $f(x)$ and the quantity $\Phi(x)$ are polynomials.

\subsubsection{SOS relaxations}
\label{sec: time avg bounds SOS}

The minimax problem on the right-hand sides of \cref{eq: weak duality,eq: strong duality} can be rewritten as a minimization subject to a pointwise inequality constraint,
\beq
\label{eq: weak duality 2}
\inf_{V \in \mathcal{C}^1(\cB)} \sup_{x \in \cB} \left[ \Phi(x) + f(x)\cdot\nabla V(x) \right] 
= \inf_{V \in\cC^1(\cB)} \U \quad \text{s.t.}  \quad S(x) \geq 0~~\forall x\in\cB,
\eeq
where
\beq
\label{eq: S}
S(x) = \U - \Phi(x) - f(x)\cdot\nabla V(x).
\eeq
In the case where the given $f$ and $\Phi$ are polynomial in the components of $x$, if $V$ is chosen to be polynomial also, then so is $S$. Even for polynomial $S$ the right-hand minimization in \cref{eq: weak duality 2} is often intractable; deciding nonnegativity of a multivariate polynomial on a given set has NP-hard computational complexity in general~\cite{Murty1987}. However, there is a standard way to enforce the nonnegativity constraint on $S$ using SOS conditions that are more tractable.

The simplest way to enforce nonnegativity of $S$ using SOS constraints is to require that $S$ admits an SOS representation, meaning that there exist a finite number of polynomials $q_k(x)$ such that $S=\sum_{k=1}^Kq_k^2$. In other words, $S$ belongs to the set $\Sigma_n$ of $n$-variate SOS polynomials, which is a strict subset (when $n\ge2$) of the set of globally nonnegative polynomials. In the corresponding SOS relaxation of the right-hand side of \cref{eq: weak duality 2}, the set $\cC^1(\cB)$ is replaced by its subset $\R[x]$ of real polynomials in $x$, and the constraint is strengthened to $S\in\Sigma_n$. In some applications~\cite{Goluskin2018a, Goluskin2019, Lakshmi2020, Olson2020} this simple SOS relaxation has given an upper bound not only equal to the left-hand side of \cref{eq: weak duality 2} but also to the maximal time average $\sup_{x_0\in\cB_0}\ov\Phi$ that one seeks to bound. In general, however, enforcing $S\in\Sigma_n$ can strictly increase the right-hand infimum in \cref{eq: weak duality 2}. This would be the case in our application to LEs below, where nonnegativity of a polynomial must be enforced only when one of its vector arguments has unit length. Thus a constraint weaker than $S\in\Sigma_n$ is needed.

When $\cB\subsetneq\R^n$ there is a standard way to formulate SOS constraints that imply nonnegativity on $\cB$ but not on all of $\Rm$, assuming that $\cB$ is a semialgebraic set. That is, we assume $\cB$ can be specified by a finite number of polynomial inequalities and/or equalities:
\beq
\label{eq: semialgebraic set}
\B = \left\{x\in \Rm~:~g_i(x) \ge 0~\text{for }i=1,\ldots,I,~\;h_j(x) = 0~\text{for }j=1,\ldots,J \right\}.
\eeq
A sufficient condition for $S$ to be nonnegative on $\cB$ is that there exist polynomials $\rho_j\in\Rx$ and SOS polynomials $\sigma_i\in\Sigma_n$ such that
\beq
\label{eq: SOS semialgebraic}
\begin{aligned}
S - \textstyle\sum_{i=1}^I\sigma_ig_i - \sum_{j=1}^J\rho_jh_j 
&\in\Sigma_n,\\[-4pt]
\sigma_i&\in\Sigma_n, \quad i=1,\ldots,I.
\end{aligned}
\eeq
When $x\in\cB$, the first sum in \cref{eq: SOS semialgebraic} is nonnegative and the second vanishes, so $S(x)\ge0$. When $x\notin\cB$, either sign of $S$ is possible. Relaxing the right-hand minimization in \cref{eq: weak duality 2} by restricting to $V\in\R[x]$ and strengthening its constraint to \cref{eq: SOS semialgebraic} gives a possibly larger infimum,
\begin{multline}
\label{eq: weak duality 3}
\inf_{V \in \mathcal{C}^1(\cB)} \sup_{x \in \cB} \left[ \Phi(x) + f(x)\cdot\nabla V(x) \right] \\[-10pt]
\le \inf_{\substack{V,\sigma_i,\rho_j \in \Rx }} \U \quad
\text{s.t.} \quad % \cref{eq: SOS semialgebraic}.
\begin{array}{rl} \\
 S - \sum_{i=1}^I\sigma_ig_i  - \sum_{j=1}^J\rho_jh_j  \in \Sigma_n~ \\
\sigma_i \in \Sigma_n, & i=1,\ldots,I.
\end{array}
\end{multline}

It is proved in~\cite{Lakshmi2020} that \cref{eq: weak duality 3} is an equality if the semialgebraic specification \cref{eq: semialgebraic set} of $\cB$ satisfies a technical condition called the Archimedean property \cite[definition 3.4]{Lasserre2008}. Every set with an Archimedean specification is compact, and every compact semialgebraic set’s specification can be made Archimedean, if it isn’t already, by adding a constraint $g_i(x)=R^2-|x|^2\ge0$ with $R$ large enough so that the constraint is redundant.

If a compact set $\cB$ is specified in an Archimedean way so that \cref{eq: weak duality 3} is an equality, and if $\cB$ is forward invariant so that the equality \cref{eq: strong duality} holds, then the maximal time average we seek to bound is equal to the infimum on the right-hand side of~\cref{eq: weak duality 3} \cite[theorem 1]{Lakshmi2020}:
\beq
\label{eq:constrained SOS program}
\sup_{x_0 \in \cB} \overline \Phi = \inf_{\substack{V,\sigma_i,\rho_j \in \Rx }} \U \quad
\text{s.t.} \quad \begin{array}{rl} \\
 S - \sum_{i=1}^I\sigma_ig_i  - \sum_{j=1}^J\rho_jh_j  \in \Sigma_n~ \\
\sigma_i \in \Sigma_n, & i=1,\ldots,I.
\end{array}
\eeq
Theorem 4 in~\cite{Korda2021} implies an equality similar to~\cref{eq:constrained SOS program} but where the infimum subject to SOS constraints is replaced by its convex dual -- a maximization subject to constraints on moments of invariant measures.

One more relaxation of the right-hand minimization in \cref{eq:constrained SOS program} is needed to obtain a computationally tractable problem: one chooses finite polynomial spaces from which to seek $V$ and each $\sigma_i$ and $\rho_j$. This yields an SOS program---a convex optimization problem that has SOS constraints and whose tunable parameters appear only affinely in these constraints and in the optimization objective~\cite{Parrilo2013a}. (Note that the tunable parameters in \cref{eq:constrained SOS program} are $\U$ and the coefficients of $V$, which appear affinely in $S$, along with the coefficients of each $\sigma_i$ and $\rho_j$.) The resulting SOS program will be computationally tractable when the dimensions of the ODE state space and of the vector spaces of tunable polynomials are not too large. A sequence of SOS programs with increasing computational cost can be defined by choosing successively higher-dimensional spaces over which to optimize the tunable polynomials. The infima of these SOS programs converge from above to the right-hand side of \cref{eq:constrained SOS program}. In cases where the equality \cref{eq:constrained SOS program} holds, this means that SOS programs can give arbitrarily sharp upper bounds on the left-hand maximal time average.

\subsection{Upper bounds on maximal LEs}
\label{sec: SOS LE}

The framework of \cref{sec: SOS time avg} gives bounds on a time-integrated average $\ov\Phi$ over bounded trajectories, where $\Phi$ is an explicit function of the state space variable. We can apply this framework to the bounding of LEs \cref{eq: LE def} by expressing the LE $\mu$ as the time-integrated average of a particular $\Phi$. This is not possible with $\Phi$ depending only on the state space vector $x$ of the ODE system \cref{eq: dynamical system}, but it is straightforward if $\Phi$ is a function on the enlarged state space $(x,y)$ since the definition \cref{eq: LE def} of $\mu$ can be rewritten as
\beq
\label{FTC non normalized}
\mu(x_0,y_0)  
% =  \limsup_{T \rightarrow \infty} \frac{1}{T} \log \norm{y(T)} 
= \limsup_{T \rightarrow \infty} \frac{1}{T}  \int_0^T  \frac{\rm d}{{\rm d}t} \log{\norm{y}} \; {\rm d}t 
 = \ov{y\cdot\tfrac{{\rm d}y}{{\rm d}t}/\norm{y}^2}
 = \ov{y^\T \Df \bigl ( x \bigr)  y /\norm{y}^2}.
\eeq
The first equality follows from the definition \cref{eq: LE def} by the fundamental theorem of calculus, relying on the absolute continuity of $t\mapsto \log\norm{y(t)}$. The third equality follows from the $y$ ODE \cref{eq: linearized system}.

In cases where $\mu(x_0,y_0)$ is positive, $y(t)$ is unbounded despite $x(t)$ being bounded, and so the framework of \cref{sec: SOS time avg} does not apply to trajectories in $(x,y)$ state space. We therefore project the tangent dynamics \cref{eq: linearized system} onto $\Sn$ by letting $z(t)=y(t)/|y(t)|$. Then $\mu$ can be written as a time average over a bounded trajectory in $(x,z)$ state space as
\begin{align}
\label{integral LE z}
\mu(x_0,z_0) & = \ov{z^\T \Df(x)z}. 
\end{align}
Versions of the above formula have appeared in the literature for decades, including with modifications for stochastic dynamics as the Furstenberg--Khasminskii formula~\cite{Arnold1995, Bedrossian2022}.

\subsubsection{Bounds from minimization over auxiliary functions}
\label{sec: LE variational}

To bound the time average on the right-hand side of \cref{integral LE z} using the framework of \cref{sec: SOS time avg}, we need evolution equations for the $(x,z)$ state space. The ODE for $z$ is found by taking the time derivative of the definition $z(t)=y(t)/|y(t)|$ and applying the $y$ ODE~\cref{eq: linearized system}. Coupling this with the $x$ ODE \cref{eq: dynamical system} gives
\begin{align}
\label{eq: xz ODE}  
\ddt \begin{bmatrix} x \\ z \end{bmatrix} 
= \begin{bmatrix} f(x) \\ \ell(x,z) \end{bmatrix}
,\qquad (x,z) \in \Rm\times\Sn,
\end{align}
where
\beq
\label{eq: ell}
\ell(x,z) = \Df(x) z - [z^\T \Df(x) z] z.
\eeq
Recall that the framework of \cref{sec: SOS time avg} was described for an ODE with state space $\Rm$, right-hand side $f(x)$, and initial set $\cB_0$ whose trajectories eventually remain in $\cB$. To bound the maximal LE of the $x$ ODE, we apply the same framework to an ODE with state space $\Rm\times\Sn$, right-hand side as in \cref{eq: xz ODE}, and initial set $\cB_0\times\Sn$ whose trajectories eventually remain in $\cB\times\Sn$. We choose $\Phi(x,z)=z^\T \Df(x)z$, so that $\mu=\ov\Phi$. Then the left-hand side of \cref{eq: weak duality} becomes the maximal LE $\mu^*_{\cB_0}$ defined in \cref{maximal LE}, and the right-hand side of \cref{eq: weak duality} becomes the upper bound that we seek. This upper bound on $\mu^*_{\cB_0}$ is stated in the first part of \cref{prop: LE duality} below, which requires no further proof. When $\cB$ is compact and forward invariant for the $x$ ODE, so that $\cB\times\Sn$ is compact and forward invariant for \cref{eq: xz ODE}, the equality \cref{eq: strong duality} proved in~\cite{Tobasco2018} can be applied in the present context to give the second part of the proposition. The assumption $f\in\cC^2$ may be stronger than needed but suffices to ensure $(f,\ell)\in\cC^1$, in which case the result of~\cite{Tobasco2018} is directly applicable to the ODE with right-hand side $(f,\ell)$. 

\begin{prop}
\label{prop: LE duality}
Let $\ddt x(t)=f(x(t))$.
\begin{enumerate}
\item Let $f\in\cC^1(\Rm,\Rm)$, and let $\cB_0\subset \Rm $ and $\cB\subset \Rm$ be such that each trajectory $x(t)$ with $x_0\in\cB_0$ is bounded for $t\ge0$ and eventually remains in $\cB$. Then the maximal LE $\mu^*_{\cB_0}$ among trajectories with initial conditions in $\cB_0$ is bounded above by
\beq
\label{eq: LE weak duality}
\mu_{\cB_0}^* \leq  \inf_{V \in \mathcal{C}^1(\cB \times \mathbb{S}^{n-1})}  
\sup_{\substack{x\in\cB~~\\ ~\;z \in \mathbb{S}^{n-1}}}
\left( z^\T \Df z + f\cdot \nabla_x V + [\Df z  - (z^\T \Df z) z]\cdot \nabla_z V \right).
\eeq
\item Let $f\in\cC^2(\Rm,\Rm)$, and let $\cB\subset\R^n$ be compact and forward invariant. Then,
\beq
\label{eq: LE strong duality}
\mu_{\cB}^* =  \inf_{V \in \mathcal{C}^1(\cB \times \mathbb{S}^{n-1})}  
\sup_{\substack{x\in\cB~~\\ ~\;z \in \mathbb{S}^{n-1}}}
\left( z^\T \Df z + f\cdot \nabla_x V + [\Df z  - (z^\T \Df z) z]\cdot \nabla_z V \right).
\eeq
\end{enumerate}
\end{prop}

In the proposition, $\nabla_xV$ and $\nabla_zV$ denote gradients of the auxiliary function $V(x,z)$ with respect to $x$ and $z$, respectively. In the first part of the proposition, neither $\cB_0$ nor $\cB$ is necessarily bounded---for instance, \cref{eq: LE weak duality} holds with $\cB_0=\cB=\Rm$ if each trajectory $x(t)$ is bounded forward in time. Note that the right-hand sides of \cref{eq: LE weak duality,eq: LE strong duality} are identical, but \cref{eq: LE strong duality} requires stronger assumptions on $\cB$ along with choosing $\cB_0=\cB$ on the left-hand side. A result analogous to \cref{eq: LE strong duality} in the context of linear systems with stochastic forcing is proved in~\cite{Kuntz2016}.

\subsubsection{Bounds from SOS relaxations}
\label{sec: LE SOS}

If the $x$ ODE \cref{eq: dynamical system} is polynomial, then so are the $(x,z)$ ODE \cref{eq: xz ODE} and the quantity $\Phi=z^\T \Df(x)z$, and we let $V(x,z)$ be polynomial also. The inner supremum on the right-hand sides of \cref{eq: LE weak duality,eq: LE strong duality} is bounded above by a constant $\U$ if and only if the polynomial
\begin{align}
\label{eq: P def}
P(x,z) = \U-z^\T \Df(x) z - f(x)\cdot\nabla_x V(x,z) - \big[\Df(x) z - ( z^\T \Df(x) z)z\big] \cdot\nabla_z V(x,z) 
\end{align}
is nonnegative for all $(x,z)\in\cB\times\Sn$. Note that the $P(x,z)\ge0$ constraint in the context of the $(x,z)$ ODE \cref{eq: xz ODE} is the analogue of the $S(x)\ge0$ constraint appearing in \cref{eq: weak duality 2} in the context of the $x$ ODE alone. Just as $S(x)\ge0$ on $\cB$ if the SOS conditions in 
\cref{eq: SOS semialgebraic} hold, $P\ge0$ on $\cB\times\Sn$ if there exist polynomials $\rho_j(x,z)$ and SOS polynomials $\sigma_i(x,z)$ such that
\beq
\label{eq: SOS semialgebraic P}
\begin{aligned}
P - \textstyle\sum_{i=1}^I\sigma_ig_i - \sum_{j=0}^J\rho_jh_j 
&\in\Sigma_{2n},\\[-4pt]
\sigma_i&\in\Sigma_{2n}, \quad i=1,\ldots,I,
\end{aligned}
\eeq
where $h_0(z)=1-|z|^2$ so that it vanishes on $\Sn$, and the conditions $g_i(x)\ge0$ and $h_j(x)=0$ for $i,j\ge1$ specify $\cB$ as in \cref{eq: semialgebraic set}. Note that we allow non-compact $\cB$, including the case $\cB=\Rm$ wherein \cref{eq: SOS semialgebraic P} should be interpreted as the single SOS constraint $P-\rho_0h_0\in\Sigma_{2n}$. 

The SOS relaxation of the right-hand minimization in \cref{eq: LE weak duality,eq: LE strong duality} is the minimization of the constant $B$ that appears in $P(x,z)$, where the polynomials $V,\sigma_i,\rho_j$ are optimized subject to \cref{eq: SOS semialgebraic P}. The infimum of the relaxed problem is an upper bound on the right-hand infimum in \cref{eq: LE weak duality}, which gives the first part of \cref{prop: LE duality SOS} below with no further proof needed. If the specification of $\cB$ is Archimedean, then adding the constraint $h_0(z)=0$ gives an Archimedean specification of $\cB\times\Sn$. In this case the SOS relaxation does not change the value of the right-hand infimum in \cref{eq: LE weak duality,eq: LE strong duality}, as explained after \cref{eq: weak duality 3}. The second part of \cref{prop: LE duality SOS} thus follows from the second part of \cref{prop: LE duality}.

\begin{prop}
\label{prop: LE duality SOS}
Let $\ddt x(t)=f(x(t))$ with $f\in\Rm[x]$. Let $\cB\subset \Rm$ be either all of $\Rm$ or specified by a finite number of polynomial inequalities $g_i(x)\ge0$ and equalities $h_j(x)=0$ as in \cref{eq: semialgebraic set}. Let $h_0(z)=1-|z|^2$, and define $P(x,z)$ as in \cref{eq: P def}.
\begin{enumerate}
\item Let $\cB_0\subset \Rm$ be such that each trajectory $x(t)$ with $x_0\in\cB_0$ is bounded for $t\ge0$ and eventually remains in $\cB$. Then the maximal LE $\mu^*_{\cB_0}$ among trajectories with initial conditions in $\cB_0$ is bounded by
\beq
\label{eq: LE weak duality SOS}
\mu_{\cB_0}^* \leq \inf_{\substack{V,\sigma_i,\rho_j \in \Rxz }} \U \quad
\text{s.t.} \quad \begin{array}{rl} \\
 P - \sum_{i=1}^I\sigma_ig_i - \sum_{j=0}^I\rho_jh_j \in \Sigma_{2n}~ \\
\sigma_i \in \Sigma_{2n}, & i=1,\ldots,I.
\end{array}
\eeq
\item Let $\cB$ be compact and forward invariant, and suppose its semialgebraic specification is Archimedean. Then,
\beq
\label{eq: LE strong duality SOS}
\mu_{\cB}^* = \inf_{\substack{V,\sigma_i,\rho_j \in \Rxz }} \U \quad
\text{s.t.} \quad \begin{array}{rl} \\
 P - \sum_{i=1}^I\sigma_ig_i - \sum_{j=0}^I\rho_jh_j \in \Sigma_{2n}~ \\
\sigma_i \in \Sigma_{2n}, & i=1,\ldots,I.
\end{array}
\eeq
\end{enumerate}
\end{prop}

The right-hand sides of \cref{eq: LE weak duality SOS,eq: LE strong duality SOS} are identical, but \cref{eq: LE strong duality SOS} requires stronger assumptions on $\cB$ along with choosing $\cB_0=\cB$ on the left-hand side.
The minimization on the right-hand sides of \cref{eq: LE weak duality SOS,eq: LE strong duality SOS} is over $V(x,z)$, $\rho_j(x,z)$, and $\sigma_i(x,z)$ in the infinite-dimensional space $\Rxz$ of $2n$-variate polynomials. Restricting each of these tunable polynomials to a finite-dimensional subspace---i.e., choosing a finite basis with tunable coefficients---gives an SOS program that can be solved computationally when $n$, $I$, $J$, and the polynomial subspace dimensions are not too large. As required for an SOS program, all tunable coefficients appear affinely in the SOS expressions since $B$ and $V$ appear affinely in the definition \cref{eq: P def} of $P$. Although restricting the tunable polynomials $V,\sigma_i,\rho_j$ to finite subspaces may increase the infimum on the right-hand sides of \cref{eq: LE weak duality SOS,eq: LE strong duality SOS}, one can always enlarge these subspaces and solve another SOS program with greater computational cost. In this way, a sequence of increasingly expensive SOS programs will give a nonincreasing sequence of infima that converge to the right-hand infimum in \cref{eq: LE weak duality SOS,eq: LE strong duality SOS}. Thus, when the equality \cref{eq: LE strong duality SOS} holds, SOS computations can give arbitrarily sharp upper bounds on the maximal LE $\mu^*_\cB$, at least in theory. The computational examples presented below in \cref{sec: examples} show that convergence of upper bounds to the maximal LE indeed can be achieved in practice.

\begin{remark}
\label{rem: general norm}
In the definition \cref{eq: LE def} of the LE, the Euclidean norm can be replaced by any other norm without changing the value of the LE~\cite{Meiss2017}. If one chooses a weighted Euclidean norm $|y|_\sfA=\big( y^\T \sfA y \big)^{1/2}$ for any positive definite $n\times n$ matrix $\sfA$, all results of \cref{sec: SOS LE} can be generalized analogously. In this generalization, $\Phi$ and $\ell$ are replaced by $\Phi_\sfA(x,z) = z^\T \sfA \Df(x) z$ and $\ell(x,z) = \Df(x)z-\Phi_\sfA(x,z) z$, and the normalization of $z$ requires $h_0(z)=1-|z|_\sfA^2=0$. The resulting SOS minimization on the right-hand sides of \cref{eq: LE weak duality SOS,eq: LE strong duality SOS} is modified only by the matrix $\sfA$ appearing affinely in $P$ and in $h_0$. For an SOS program defined by fixing the spaces over which $V,\sigma_i,\rho_j$ are optimized, different fixed $\sfA$ may give smaller or larger minima, meaning better or worse upper bounds on the maximal LE. One might tune $\sfA$ along with $V,\sigma_i,\rho_j$ to minimize the bound. However, doing so would make the minimization non-convex because $\sfA$ and $\rho_0$ multiply each other in the first SOS constraint, so we do not pursue this idea here.
\end{remark}

\subsection{Symmetries}
\label{sec: sym}

Making use of symmetries in SOS programming significantly reduces computational cost and numerical imprecision. In particular, when a polynomial that is constrained to be SOS has a known symmetry, block diagonal structure can be imposed on the matrix representation of that polynomial in the corresponding semidefinite program~\cite{Gatermann2004}. For the SOS programs described in \cref{sec: SOS time avg} in the general setting of bounding time averages in ODEs, the SOS programs can be formulated to have symmetries if the ODEs and the quantities whose time averages are being bounded have corresponding symmetries~\cite{Goluskin2019, Lakshmi2020}. Since our framework for bounding maximal LEs is a particular application of \cref{sec: SOS time avg}, we can use any symmetries that are shared by the $(x,z)$ ODE \cref{eq: xz ODE}, the domain $\cB\times\Sn$, and the function $\Phi=z^\T\Df(x)z$ whose time average is equal to an LE. \Cref{sec: sym dynamics} describes how orthogonal symmetries of the $x$ dynamics induce symmetries of the $(x,z)$ dynamics, and \cref{sec: sym opt} describes how the latter lead to symmetries in the SOS programs that give upper bounds on LEs.

\subsubsection{Symmetries of the $(x,z)$ dynamics}
\label{sec: sym dynamics}

An ODE is said to be symmetric under a transformation if solutions are mapped to solutions. Let $\Lambda\in GL(n)$, where $GL(n)$ denotes the group of invertible linear transformations on $\Rm$. An ODE $\ddt x=f(x)$ on $\R^n$ is symmetric under $\Lambda$ if and only if $f$ is \emph{equivariant} under $\Lambda$, meaning $f(\Lambda x)=\Lambda f(x)$ for all $x\in\R^n$. A function $\Phi$ is \emph{invariant} under $\Lambda$ if $\Phi(\Lambda x)=\Phi(x)$ for all $x\in\R^n$. The invariant transformations of any function form a group, as do its equivariant transformations. A function is said to be $\cG$-invariant (resp.\ $\cG$-equivariant) for a group $\cG$ if it is invariant (resp.\ equivariant) under all $\Lambda\in\cG$. Turning particularly to the $(x,z)$ dynamics~\cref{eq: xz ODE}, we seek a symmetry group for which the right-hand side of the ODE is equivariant, and the function $\Phi=z^\T \Df(x)z$ is invariant, under all transformations in the group.

The transformation $(x,z)\mapsto (x,-z)$ is a symmetry both of the $(x,z)$ ODE \cref{eq: xz ODE} and of the function $\Phi(x,z)=z^\T \Df(x)z$, reflecting the fact that the growth or decay rate of a vector in the tangent space is unaffected by reversing its direction. The equivariance of the ODE right-hand side under this transformation amounts to the oddness $\ell(x,-z)=-\ell(x,z)$, where $\ell(x,z)$ is defined as in \cref{eq: ell}, and the invariance $\Phi(x,-z)=\Phi(x,z)$ is clear.

Further symmetries of the $(x,z)$ dynamics are guaranteed only if the $x$ ODE alone is symmetric under orthogonal transformations, common examples of which include sign changes or rotations. In particular, if the $x$ dynamics are symmetric under $\Lambda\in O(n)$, where $O(n)$ denotes the group of orthogonal linear transformations on $\Rm$, then the $(x,z)$ dynamics are symmetric under $(x,z)\mapsto(\Lambda x,\Lambda z)$, and so is the function $\Phi(x,z)$. To state this precisely in the following proposition, for each group $\cG\subset O(n)$ we define the corresponding group $\cG'\in O(2n)$ as
\beq
\label{eq: GG}
\cG' = \left\{\begin{bmatrix}\Lambda&0\\0&\pm\Lambda\end{bmatrix}\in O(2n)~:~\Lambda\in\cG\subset O(n) \right\}.
\eeq

\begin{prop}
\label{prop: orthogonal sym}
Let $f\in\mathcal{C}^1(\Rm,\Rm)$. Define $\ell:\Rm\times\Sn\to\Rm$ and $\Phi:\Rm\times\Sn\to\R$~by $\ell(x,z) = \Df(x) z - [z^\T \Df(x) z] z$ and $\Phi(x,z)= z^\T \Df(x)z$. If $f$ is $\cG$-equivariant for some $\cG\subset O(n)$, then $(f,\ell)$ is $\cG'$-equivariant for $\cG'$ defined by~\cref{eq: GG}, and $\Phi$ is $\cG'$-invariant.
\end{prop}

\begin{proof}
Let $\Lambda\in\cG$. Invariance under $\cG'$ is equivalent to invariance under $(x,z)\mapsto(x,-z)$ and under $(x,z)\mapsto(\Lambda x,\Lambda z)$, and likewise for equivariance. For $(x,z)\mapsto(x,-z)$, the equivariance of $(f,\ell)$ and invariance of $\Phi$ are immediate, as described above in the text. For $(x,z)\mapsto(\Lambda x,\Lambda z)$, the invariance of $\Phi$ is shown by
\begin{align}
\Phi(\Lambda x , \Lambda z) = (\Lambda z)^\T \Df(\Lambda x) \Lambda z = z^\T \Lambda^\T \Lambda \Df(x) z = z^\T \Df(x) z = \Phi(x,z),
\end{align}
where the second equality uses the relation $\Df(\Lambda x) \Lambda = \Lambda \Df(x)$ that is found by differentiating the equivariance definition $f(\Lambda x)=\Lambda f(x)$, and the third equality requires $\Lambda\in O(n)$. The claimed equivariance of $(f,\ell)$ holds if $f(\Lambda x)=\Lambda f(x)$, which is true by assumption, and also $\ell(\Lambda x,\Lambda z)=\Lambda\ell(x,z)$. To show the latter, we note that $\ell$ is related to $\Phi$ by $\ell(x,z) = \Df(x) z - \Phi(x,z) z$, and we use the relation $\Df(\Lambda x) \Lambda = \Lambda \Df(x)$ again along with the invariance of $\Phi$:
\begin{align}
\ell(\Lambda x , \Lambda z) = \Df(\Lambda x) \Lambda z -  \Phi(\Lambda x, \Lambda z) \Lambda z = \Lambda [ \Df(x) z -  \Phi(x,z) z ] = \Lambda \ell(x,z).
\end{align}
This completes the proof.
\end{proof}

\begin{remark}
\label{rem: similarity}
The symmetry results described above and in \cref{sec: sym opt} below are not directly applicable if $\cG\not\subset O(n)$, but they are applicable after a change of variables if $\cG$ is a subgroup of $O(n)$ after a similarity transformation. Concretely, suppose the $x$ dynamics has a symmetry group $\cG_\sfM\not\subset O(n)$ that takes the form
\beq
\label{eq: G similar}
\cG_\sfM = \{ \sfM^{-1}\Lambda \sfM~:~\Lambda\in\cG \}
\eeq
for any fixed $\cG\subset O(n)$ and $\sfM\in GL(n)$. (Groups of the form \cref{eq: G similar} are exactly all compact subgroups of $GL(n)$ \cite[theorem 0.3.5]{Bredon1972}.) Defining $\tilde x(t)=\sfM x(t)$, so that tangent space vectors for $\tilde x(t)$ are $\sfM y(t)$, we note that the $x$ dynamics and the $\tilde x$ dynamics have the same LEs. Thus the framework of \cref{sec: SOS LE} for bounding LEs can be applied to the $\tilde x$ ODE $\ddt \tilde x=\sfM f(\sfM^{-1}\tilde x)$, rather than to the $x$ ODE. The $\tilde x$ ODE has the orthogonal symmetry group $\cG$, so according to \cref{prop: orthogonal sym} the $(x,z)$ dynamics is symmetric under $\cG'$.
\end{remark}

\subsubsection{Symmetries of the optimization problems}
\label{sec: sym opt}

If an ODE, its domain, and a time-averaged quantity share a compact linear group of symmetries, then the upper bound on this time average given by the right-hand infimum in \cref{eq: weak duality,eq: strong duality} is unchanged if the auxiliary function $V$ is further constrained to be invariant under the same symmetries. In particular, any $V$ that yields an upper bound on a time average can be used to construct a symmetrized $V$ yielding the same bound. This is proved in \cite[Proposition A.1]{Goluskin2019} in the context of a finite cyclic symmetry group. \Cref{prop: sym V} in the appendix generalizes the result to all compact groups of linear transformations---i.e., all groups of orthogonal transformations, along with similarity transformations \cref{eq: G similar} of such groups.

We apply these symmetry results in the context of bounding LEs by the infimum on the right-hand sides of \cref{eq: LE weak duality,eq: LE strong duality}, in which case the pertinent symmetry group $\cG'$ of the $(x,z)$ dynamics is given by \cref{prop: orthogonal sym} above. Provided that the domain $\cB$ is $\cG$-invariant, the following proposition guarantees that corresponding symmetries can be imposed on $V(x,z)$ in the right-hand minimization of \cref{eq: LE weak duality} and \cref{eq: LE strong duality} without changing the resulting upper bound.

\begin{prop}
\label{prop: sym V LE}
Let $\cG\subset O(n)$ and $f\in\mathcal{C}^1(\cB,\Rm)$. If $f$ is $\cG$-equivariant and $\cB$ is $\cG$-invariant, then the infimum over $V\in\mathcal C^1(\cB\times\Sn)$ on the right-hand sides of \cref{eq: LE weak duality,eq: LE strong duality} is unchanged if $V$ is constrained to be $\cG'$-invariant for $\cG'$ defined by~\cref{eq: GG}.
\end{prop}

\begin{proof}
Under the assumptions on $\cB$ and $f$, \cref{prop: orthogonal sym} guarantees that $(f,\ell)$ is $\cG'$-equivariant and $\Phi=z^\T Df z$ is $\cG'$-invariant. The domain $\cB\times\Sn$ is also $\cG'$-invariant because $\cB$ is $\cG$-invariant by assumption and $\Sn$ is $\cG$-invariant for any $\cG\subset O(n)$. We can therefore apply \cref{prop: sym V} from the appendix in the context of symmetry group $\cG'$, domain $\cB\times\Sn$, ODE right-hand side $(f,\ell)$, and $\Phi=z^\T Df z$. In this context the proposition guarantees that if there exist $V \in \cC^1(\cB\times\Sn)$ and $\U\in\R$ such that $P(x,z) \geq 0$ for all $(x,z)\in\cB\times\Sn$, where $P$ is as defined by \cref{eq: P def}, then there exists $\widehat V(x,z) \in \cC^1(\cB\times\Sn)$ that satisfies the same inequality and is $\cG'$-invariant. Therefore, the infimum~of
\beq
\label{eq: min P constraint}
\inf_{V \in\cC^1(\cB\times\Sn)} \U \quad \text{s.t.}  \quad P(x,z) \geq 0~~\forall (x,z)\in\cB\times\Sn
\eeq
is unchanged if $V$ is further constrained to be $\cG'$-invariant. The constrained minimization problem \cref{eq: min P constraint} is equivalent to the minimax problem on the right-hand sides of \cref{eq: LE weak duality,eq: LE strong duality}, so the claim is proved.
\end{proof}

\Cref{prop: sym V LE}, which allows invariance to be imposed on $V$ in the minimization that appears on the right-hand sides of \cref{eq: LE weak duality,eq: LE strong duality}, has an analogue for the minimization's SOS relaxation that appears on the right-hand sides of \cref{eq: LE weak duality SOS,eq: LE strong duality SOS}. This is stated by \cref{prop: sym SOS LE} below, whose assumptions require that the $\cG$-invariant set $\cB$ is specified in terms of $\cG$-invariant polynomials, and whose conclusions allow invariance to be imposed not only on $V$ but also on any other tunable polynomials $\sigma_i$ and $\rho_j$. We do not include a proof because the result is a direct application of \cref{prop: sym SOS} in the same way that \cref{prop: sym V LE} is an application of \cref{prop: sym V}.

\begin{prop}
\label{prop: sym SOS LE}
Let $\cG\subset O(n)$ and $f\in\R^n[x]$. Let $\cB\subset\R^n$ be either all of $\Rm$ or specified by a finite number of polynomial inequalities $g_i(x)\ge0$ and equalities $h_j(x)=0$ as in \cref{eq: semialgebraic set}. If $f$ is $\cG$-equivariant and all $g_i$ and $h_j$ are $\cG$-invariant, then the infimum on the right-hand sides of \cref{eq: LE weak duality SOS,eq: LE strong duality SOS} over $V,\sigma_i,\rho_j\in\Rxz$ (or over subspaces of $\Rxz$) is unchanged if $V$ and all $\sigma_i$ and $\rho_j$ are constrained to be $\cG'$-invariant for $\cG'$ defined by \cref{eq: GG}. When $V$ and all $\sigma_i$ and $\rho_j$ are $\cG'$-invariant, the first expression constrained to be SOS in \cref{eq: LE weak duality SOS,eq: LE strong duality SOS} is $\cG'$-invariant also.
\end{prop}

\Cref{prop: sym SOS LE} has practical implications for the SOS programs whose solutions give upper bounds on the maximal LE---i.e., the SOS programs obtained by restricting the minimization on the right-hand sides of \cref{eq: LE weak duality SOS,eq: LE strong duality SOS} to finite spaces of polynomials. Imposing $\cG'$-invariance on $V,\sigma_i,\rho_j$ does not change the resulting upper bounds on LEs, but it ensures that all expressions constrained to be SOS are $\cG'$-invariant. Exploiting this latter invariance in the numerical solution of SOS programs can greatly improve computational cost and accuracy~\cite{Gatermann2004}.

\begin{remark}
If the $x$ ODE has no symmetries, then $\cG$ contains only the identity, and $\cG'$ contains the identity and the transformation $(x,z)\mapsto(x,-z)$. In this case \cref{prop: sym V LE} reduces to the statement that making $V(x,z)$ even in $z$ does not change the resulting upper bounds on LEs, and \cref{prop: sym SOS LE} says the same about making all tunable polynomials even in~$z$.
\end{remark}

\section{Computational examples}
\label{sec: examples}

This section presents two examples for which we have computed upper bounds on the maximal LE using SOS programming, in particular by numerically solving the minimization on the right-hand sides of \cref{eq: LE weak duality SOS,eq: LE strong duality SOS} with the tunable polynomials restricted to various maximum degrees. The example of \cref{sec: Lorenz} is the Lorenz system, which has a chaotic attractor. The example of \cref{sec: HH} is the H\'enon--Heiles system with a chosen range of energy values, which is Hamiltonian and chaotic. The computed upper bounds converge rapidly to the maximal LE as we raise the degrees of the tunable polynomials. In both examples our best bounds are sharp to at least five digits.

The SOS programs whose solutions we report below were solved numerically in the standard way, wherein each polynomial that must be SOS is represented by a symmetric matrix using monomial bases. This results in matrices subject to affine and semidefinite constraints, constituting a semidefinite program that can be solved numerically using a variety of standard software. Here we used a version~\cite{aeroimperial-yalmip} of the parser YALMIP~\cite{Lofberg2009, YALMIP2016} to specify SOS programs, translate them into semidefinite programs, and pass the latter to a solver. For the solver we used Mosek~\cite{Mosek2000}, which is not open source but implements a second-order interior-point algorithm.

\subsection{The Lorenz system}
\label{sec: Lorenz}

Consider the Lorenz system~\cite{Lorenz1963},
\beq
\label{eq: Lorenz}
\ddt\begin{bmatrix} x_1\\x_2\\x_3\end{bmatrix} = 
\begin{bmatrix} 
\sigma (x_2 - x_1 ) \\
r x_1-x_2-x_1 x_3 \\
x_1 x_2 - \beta x_3
\end{bmatrix},
\eeq
with the standard chaotic parameter values $(\beta,\sigma,r) = (8/3,10,28)$. For the unstable fixed point at the origin, or any trajectory approaching it along its stable manifold, the LEs are the real parts of the eigenvalues of $\Df(\mathbf 0)$. Thus the leading LE at the origin is $\mu_1= \bigl(-1-\sigma+\sqrt{1-2\sigma+4r \sigma + \sigma^2} \bigr )/2 \approx11.82772$. In fact this is the maximal LE among all trajectories of the Lorenz system, as verified by an upper bound on the maximal LE that is saturated by $\mu_1$ at the origin. This sharp upper bound has been proved analytically~\cite{Leonov2016a} by a method that is not sharp in general but that is sometimes sharp when  the maximal LE is attained on a fixed point. The computations we report below also give the sharp upper bound, at least up to numerical error.

We use the SOS approach of \cref{sec: LE SOS} to compute upper bounds on the maximal LE $\mu^*_{\R^3}$ of the Lorenz system. For any $\cB$ containing the global attractor, expression \cref{eq: LE weak duality SOS} gives an upper bound. We choose $\cB=\R^3$, so the constraints on the right-hand side of \cref{eq: LE weak duality SOS} reduce to simply $P-\rho_0h_0\in\Sigma_6$. For this non-compact $\cB$, part 2 of \cref{prop: LE duality SOS} does not apply to guarantee sharpness of the upper bound, nonetheless the computations described below apparently give the sharp result. The SOS programs we solve are obtained by restricting $V$ and $\rho_0$ to finite-dimensional spaces, in particular by specifying a maximum degree $d$ and imposing invariances that, according to \cref{prop: sym SOS LE}, do not change the optimum of the SOS program. In addition to the invariance under negation of $z$ that can always be imposed, the symmetry of the Lorenz equations under $(x_1,x_2,x_3)\mapsto(-x_1,-x_2,x_3)$ induces another, so $V$ and $\rho_0$ are optimized within the spaces
\beq
\label{eq: Vd}
\mathcal V_d = \left\{ p\in\R[x,z]_d ~:~ p(x,z)=p(x,-z)=p(\Lambda x,\Lambda z)~~\forall (x,z)\in\R^3\times\R^3 \right\},
\eeq
where $\Lambda$ negates the first two coordinates, and $\R[x,z]_d$ denotes the space of polynomials in $(x,z)$ with total degree no larger than $d$. In our computations we let $\rho_0\in\mathcal V_d$ for various $d$, but we keep $V\in\mathcal V_2$ since quadratic $V$ suffice for sharp bounds in this case. Each fixed $d$ gives an SOS program that we solve to obtain an upper bound on the maximal LE of the Lorenz system:
\beq
\label{eq: Lorenz SOS}
\mu_{\R^3}^* \leq \inf_{\substack{V\in\cV_2\\ \rho_0\in\cV_d}} \U
\quad \text{s.t.} \quad P-\rho_0(1-|z|^2) \in \Sigma_6,
\eeq
where $P(x,z)$ is defined as in \cref{eq: P def}.

\Cref{tab: table Lorenz bounds} reports the upper bounds \cref{eq: Lorenz SOS} found by numerically solving the SOS programs with the maximum degree of $\rho_0$ fixed to $d=2$, 4, and 6. The upper bounds are converged to at least 7 digits when $d\ge4$, and we expect that the exact minimum of the SOS program with $d=4$ is identical to the maximal LE of the Lorenz system, meaning \cref{eq: Lorenz SOS} is an equality in this case.

\begin{table}[t]
\caption{Upper bounds on the global maximal LE $\mu^*_{\R^3}$ of the Lorenz system \cref{eq: Lorenz} at the standard parameters, found by numerically solving the right-hand SOS program in \cref{eq: Lorenz SOS} with quadratic $V$ and with the maximum degree $d$ of $\rho_0$ fixed to various values. Tabulated bounds are rounded to the precision shown. To this precision, the converged bound is saturated by the fixed point at the origin (see text).}
\label{tab: table Lorenz bounds}
\begin{center}
\begin{tabular}{cc}
Degree of $\rho_0$ & Upper bound on maximal LE \\[2pt] \hline
2  & 14.02562 \\ 
4   &  11.82772 \\
6   &  11.82772 
\end{tabular}
\end{center}
\end{table}

The results of \cref{tab: table Lorenz bounds} constitute an example of our method giving sharp upper bounds on the maximal LE of a chaotic system using polynomials of modest degree---and, therefore, using SOS computations of modest cost. The example is relatively easy, even though the Lorenz system is chaotic, in the sense that the maximal LE is attained by the fixed point at the origin rather than a more complicated trajectory. In the next subsection we apply our method to a more challenging example.

\subsection{The H\'enon--Heiles system}
\label{sec: HH}

To demonstrate the success of our method when the maximal LE is attained on an orbit that is more complicated than a fixed point, we consider the H\'enon--Heiles system~\cite{Henon1964, Shevchenko2003}, which is Hamiltonian. The Hamiltonian function is 
\begin{align}
\label{eq: HH Ham}
H(x_1,x_2,x_3,x_4) = \frac{1}{2} (x_1^2 +x_2^2+x_3^2+x_4^2) + x_1^2 x_2 - \frac{1}{3} x_2^3, 
\end{align}
where $x_1$ and $x_2$ are position variables with corresponding momentum variables $x_3$ and $x_4$. This $H$ gives rise to the ODE system
\beq
\label{eq: HH}
\ddt\begin{bmatrix} x_1\\x_2\\x_3\\x_4\end{bmatrix} = 
\begin{bmatrix} 
x_3 \\
x_4 \\
-x_1 -2x_1x_2 \\
- x_2 - x_1^2 + x_2^2
\end{bmatrix}.
\eeq

For our example, we seek a set that is invariant under the dynamics and where the maximal LE is not attained by a fixed point. The fixed points of \cref{eq: HH} have energies of $H=0$ and $H = 1/6$. The three fixed points with $H=1/6$, which are related by symmetry as described below, have a large leading LE. We omit these points by restricting attention to the set where $0\le H\le1/7$, which is dynamically invariant since $H$ is conserved along trajectories. This set consists of several disconnected regions in $\R^4$, one of which is bounded; \cref{fig: HH triangle} shows the intersection of these regions with the $(x_1,x_2)$ plane. To restrict to the bounded region we add the condition $x_1^2+x_2^2\le1$, so the region of interest is the compact semialgebraic set
\beq
\label{eq: Triangle domain HH}
\B = \{ x \in \R^4~:~g_i(x) \geq 0~\text{for }i = 1,2,3 \}, 
\eeq
where
\begin{subequations}
\label{eq: HH gi}
\begin{align}
g_1(x) & = 1/7 - H(x_1,x_2,x_3,x_4), \\
g_2(x) & = H(x_1,x_2,x_3,x_4), \\
g_3(x) & = 1 - x_1^2 - x_2^2.
\end{align}
\end{subequations}
The set $\cB$ is invariant under the dynamics since the invariant region where $0\le H\le1/7$ does not intersect the surface where $g_3=0$. We have verified using SOS computations (although not analytically) that the semialgebraic specification of $\cB$ is Archimedean, in which case part 2 of \cref{prop: LE duality SOS} guarantees that \cref{eq: LE strong duality SOS} holds, so SOS programs with increasing polynomial degrees give arbitrarily sharp upper bounds on $\mu^*_\cB$.

\begin{figure}[t]
\begin{center}
\includegraphics[scale=.50, trim={4pt 24pt 30pt 48pt}, clip]{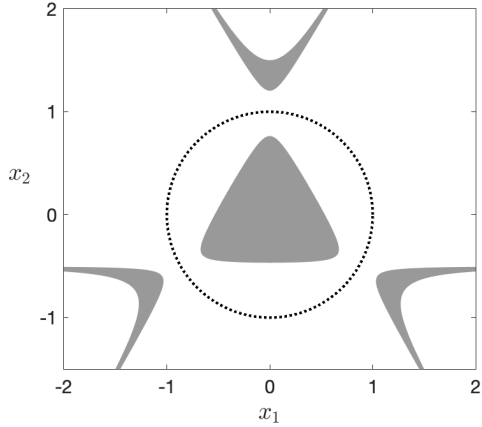}
\end{center}
\caption{The shaded areas show the intersection of the $(x_1,x_2)$ plane with the regions where $0 \leq H \leq 1/7$. These disconnected regions are separated by the set where $x_1^2+x_2^2=1$ (\shortdottedrule). The central shaded area is where the compact region $\cB$ intersects the $(x_1,x_2)$ plane.}
\label{fig: HH triangle}
\end{figure}

We use the SOS approach of \cref{sec: LE SOS} to compute upper bounds on the maximal LE $\mu^*_{\cB}$ among trajectories in the dynamically invariant set $\cB$. In order to impose symmetries on the SOS programs as described in \cref{sec: sym opt}, we note that the H\'enon--Heiles ODE has a symmetry group $\cG\subset O(4)$. Each $\Lambda\in\cG$ is a transformation in which some $A\in D_3$ acts on both the position vector $(x_1,x_2)$ and the momentum vector $(x_3,x_4)$, where $D_3$ is the dihedral group generated by two planar transformations: reflection in the first coordinate and rotation by $2\pi/3$. Since the $g_i$ polynomials defining $\cB$ are $\cG$-invariant also, \cref{prop: sym SOS LE} guarantees that the upper bound from the SOS program is unchanged if $\cG'$-invariance is imposed on all tunable polynomials. However, the software~\cite{aeroimperial-yalmip} with which we implement SOS programs can automatically exploit sign-reversal symmetries but not rotational ones, so the only symmetry of \cref{eq: HH} that we make use of is $(x_1,x_2,x_3,x_4)\mapsto(-x_1,x_2,-x_3,x_4)$. The tunable polynomials are therefore optimized within the spaces
\beq
\label{eq: Wd}
\mathcal W_d = \left\{ p\in\R[x,z]_d ~:~ p(x,z)=p(x,-z)=p(\Lambda x,\Lambda z)~~\forall (x,z)\in\R^4\times\R^4 \right\},
\eeq
where $\Lambda$ negates the first and third coordinates. In our computations we specify the same maximum total degree $d$ for all tunable polynomials. Each fixed $d$ gives an SOS program that we solve to obtain an upper bound on the maximal LE of the H\'enon--Heiles system in the region~$\cB$:
\beq
\label{eq: HH SOS}
\mu_{\cB}^* \leq \inf_{V,\sigma_1,\sigma_2,\sigma_3,\rho_0\in\mathcal W_d} \U
\quad \text{s.t.} \quad 
\begin{array}[t]{rl}P-\sum_{i=1}^3\sigma_ig_i -\rho_0(1-|z|^2) \in \Sigma_8~\\
\sigma_i\in\Sigma_8, & i=1,2,3,
\end{array}
\eeq
where $P(x,z)$ is defined as in \cref{eq: P def}.

\Cref{tab: SOS bounds HH} reports the upper bounds \cref{eq: HH SOS} found by numerically solving the SOS program with the maximum degrees of all tunable polynomials fixed to $d=2$, 4, 6, 8, and 10. The upper bounds improve as $d$ is raised, appearing to converge to five digits once polynomials are of degree 8. To confirm the sharpness of this upper bound, we computed a number of unstable periodic orbits of the H\'enon--Heiles system and calculated the leading LE of each using a procedure described below. On the shortest of these periodic orbit we computed a leading LE of $\mu_1\approx0.23081$, confirming that the best upper bounds in \cref{tab: SOS bounds HH} are sharp to at least five digits.

\begin{table}[t]
\caption{Upper bounds on the maximal LE $\mu^*_\cB$ among trajectories of the H\'enon--Heiles system \cref{eq: HH} in the set $\cB$ defined by \cref{eq: Triangle domain HH,eq: HH gi}. Bounds are found by numerically solving the right-hand SOS program in \cref{eq: HH SOS} with the maximum degree $d$ of all tunable polynomials fixed to various values. Tabulated bounds are rounded to the precision shown. To this precision, the best bound is saturated by the periodic orbits shown in \cref{fig: HH PO}.}
\label{tab: SOS bounds HH}
\begin{center}
\begin{tabular}{cc}
Degree of polynomials & Upper bound on maximal LE \\ 
\hline
2  & 0.86999 \\ 
4  & 0.41206 \\ 
6  & 0.26717 \\
8  & 0.23081 \\ 
10 & 0.23081
\end{tabular}
\end{center}
\end{table}

To compute a number of unstable periodic orbits---and thus to find one with a large leading LE---we first numerically integrated the H\'enon--Heiles system from various initial conditions and searched for close returns to the initial conditions. A fourth--order symplectic integrator~\cite{Donnelly2005} was used to conserve $H$. The numerical coefficients for the scheme were as in~\cite[equation 4.8]{Forest1990}. Choosing 121 initial conditions along the curve where $(x_3,x_4)=(0,0)$ and $H=1/7$, we integrated each trajectory for $0\le t\le33$. For every trajectory that satisfied the close return condition $|x(t)-x(0)|<10^{-3}$ at some $t\ge1$, we selected the smallest such time $T$ and corresponding initial condition $x(0)$ as producing an approximate periodic orbit. 

After compiling 19 pairs of initial conditions and periods giving approximate periodic orbits, we used a shooting method to converge each more precisely to a periodic orbit. Precise convergence was needed because errors in an initial condition often produce larger errors in the LE later computed for that orbit. Our shooting method implementation used MATLAB's \texttt{fminsearch} function to minimize $|x(T)-x(0)|$ by tuning the period $T$ and the $x_1$ coordinate of the initial condition. (The initial $x_2$ coordinate was determined from the conditions that $(x_3,x_4)=(0,0)$ and $H=1/7$.) On each iteration of the shooting method, symplectic integration up to time $T$ was carried out using a fixed time step close to $10^{-3}$ but evenly dividing $T$. All 19 orbits were converged such that $|x(T)-x(0)|<10^{-13}$, which required modifying the default options of the \texttt{fminsearch} function to tighten tolerances and allow more iterations. Among the 19 converged orbits there were only 7 different periods (approximately 6.97, 8.07, 15.6, 23.0, 29.3, 29.7, and 30.0) with orbits of the same period being related to each other by symmetries of the H\'{e}non--Heiles system. The leading LEs of all 7 orbits were computed as described below, and the largest LE was found on the orbit with the shortest period, whose numerically converged initial condition and period~are
\beq
x(0) = (0.562878385826716,-0.053847890920149,0,0), \quad T = 6.966517640959103.
\eeq
\Cref{fig: HH PO} shows this periodic orbit, along with the two symmetry-related orbits obtained by rotating both the position vector $(x_1,x_2)$ and the momentum vector $(x_3,x_4)$ by~$\pm2\pi/3$.

\begin{figure}[t]
\begin{center}
\includegraphics[scale=.5, trim={0 0 0 0}, clip]{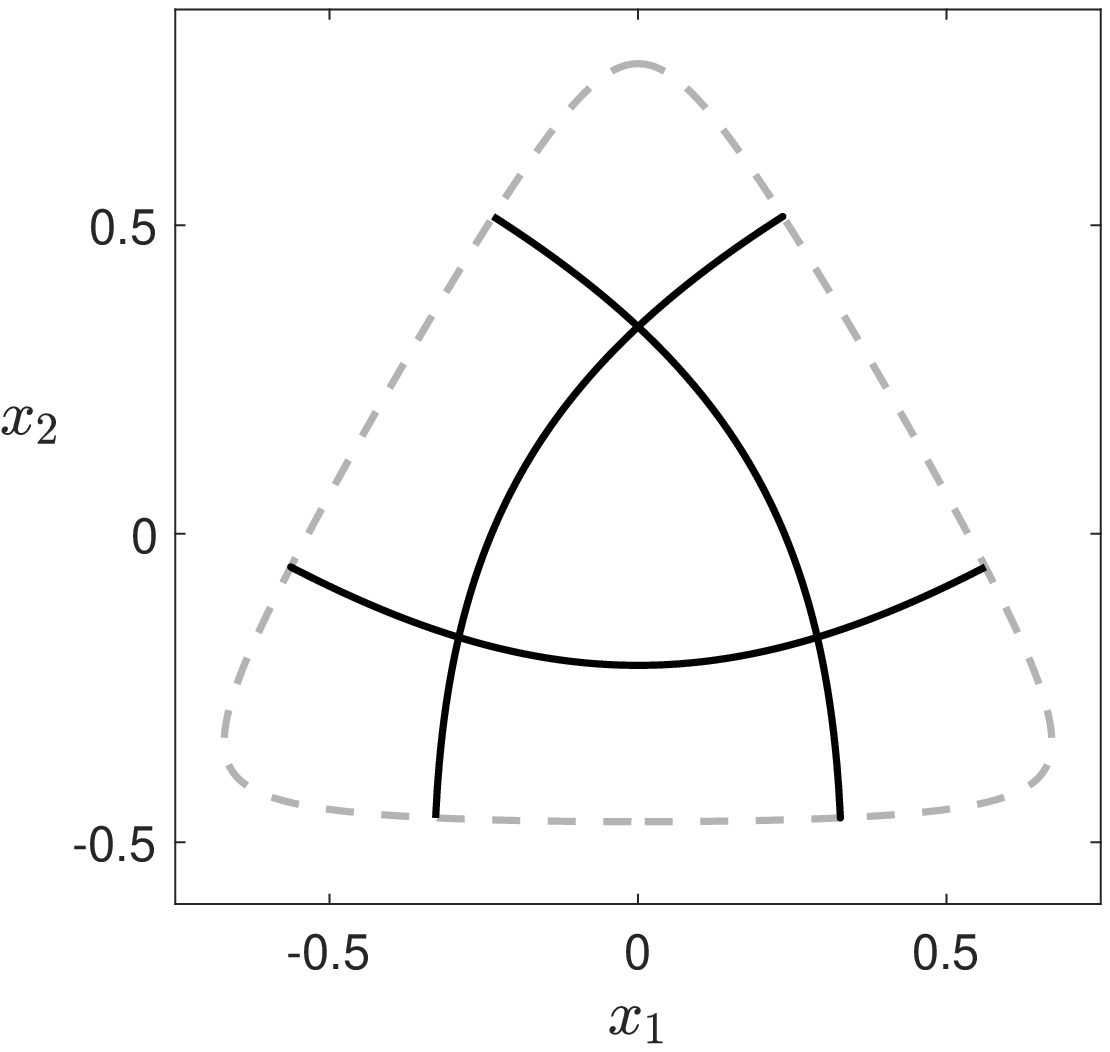}
\hspace{20pt}
\includegraphics[width = 0.38\linewidth]{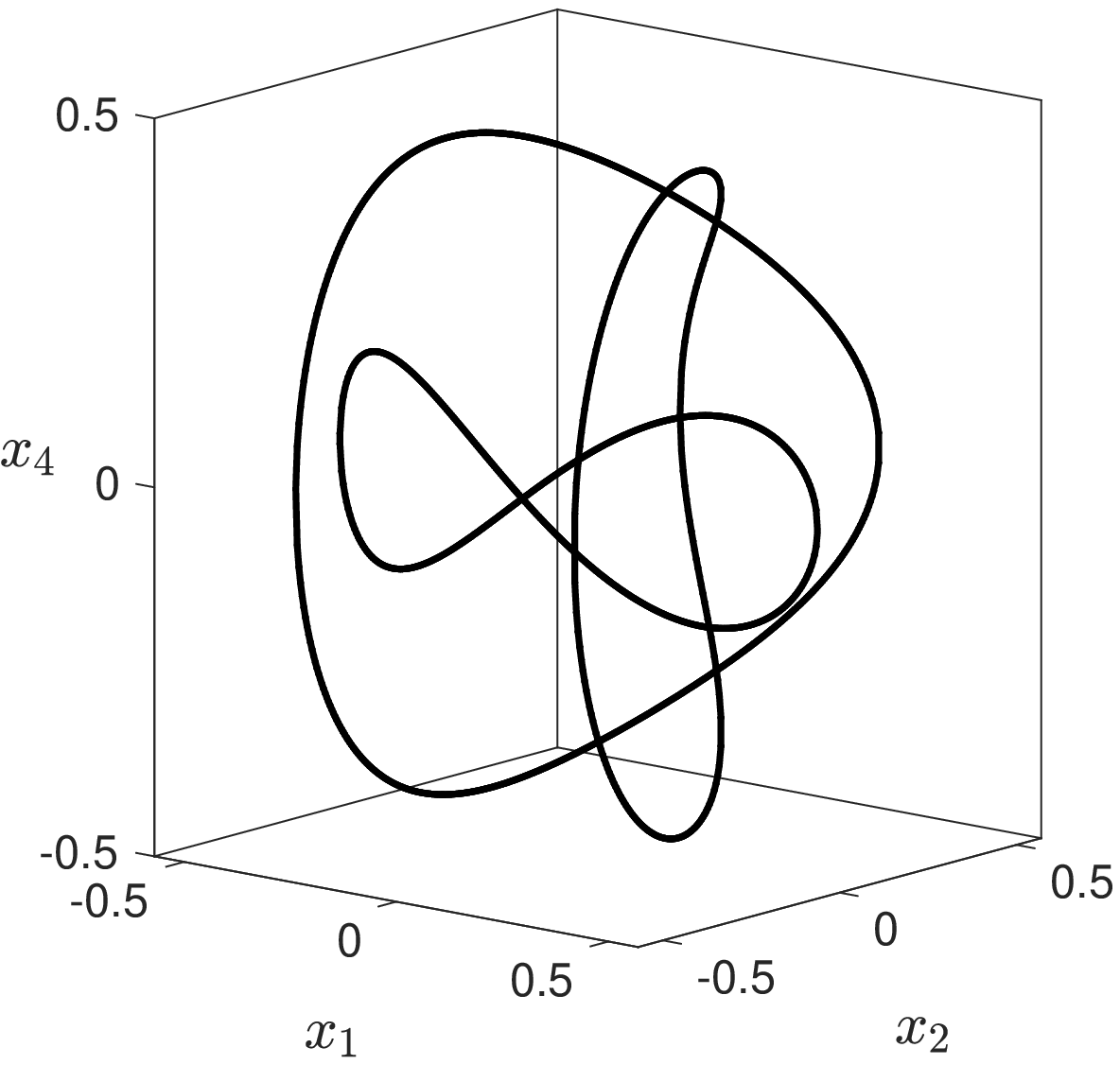}
\end{center}
\caption{Orbits of the H\'{e}non--Heiles system whose leading LE of $\mu_1\approx0.23081$ agrees with our best upper bound on $\mu^*_\cB$ to all five digits reported in \cref{tab: SOS bounds HH}. The three orbits are related by the symmetry in which both position and momentum vectors are rotated by $\pm2\pi/3$ (see text). The left panel shows their projection ($\solidrule$) onto the $(x_1,x_2)$ plane along with that of the $H=1/7$ energy surface (${\color{gray}\dashedrule}$), and the right panel shows their projection onto $(x_1,x_2,x_4)$.}
\label{fig: HH PO}
\end{figure}

To compute the leading LE on each periodic orbit we integrated the ODEs \cref{eq: dynamical system,eq: linearized system} for the orbit $x(t)$ and the unnormalized tangent space vector $y(t)$, using the symplectic integrator for $x$ and the fourth-order Runge--Kutta method for $y$. In order to align $y$ with the leading Lyapunov vector, we began with $y$ of unit length pointing in a random direction and integrated repeatedly around the orbit, normalizing $y$ once per period. After 14 periods the direction of $y$ converged to machine precision. Using this normalized $y$ for the initial condition, we integrated for one more period and calculated the leading LE as $\mu_1=\tfrac{1}{T}\log|y(T)|$. For the shortest-period orbits, which are shown in \cref{fig: HH PO}, our procedure gave a leading LE of $\mu_1\approx0.23081$. This value agrees with our best upper bound in \cref{tab: SOS bounds HH} to all five digits. (We do not compare additional digits because we estimate that the numerically computed infima of the SOS programs are precise to about five digits.) These results strongly suggest that the maximal LE among trajectories in $\cB$ indeed occurs on the orbits shown in \cref{fig: HH PO}.

We emphasize that using numerical integration to find periodic orbits and obtain the precise value $\mu_1\approx0.23081$ on one such orbit was much more difficult than using SOS optimization to compute the upper bound $\mu^*_\cB\le0.23081$. There is inherent numerical difficulty in finding unstable orbits, especially the most unstable orbits that attain the maximal LE, and in converging the orbits to the precision needed to compute their LEs. On the other hand, such considerations do not affect the SOS programs that give upper bounds on the maximal LE. Our MATLAB script implementing these SOS programs with existing software is fewer than 30 lines. For the computation with degree-8 polynomials that gives the nearly exact bound, the runtime was less than 8 minutes on a laptop and could be made faster by fully exploiting the rotational symmetries.

\section{Conclusions}
\label{sec: con}

We have described an approach for using polynomial optimization to compute upper bounds on a polynomial ODE system's maximal LE among all bounded trajectories in a given set---i.e., the leading LE of the most unstable trajectory. For trajectories remaining in a compact domain that is specified as a semialgebraic set subject to a mild technical condition, the upper bounds are guaranteed to converge to the maximal LE as the polynomial degrees in the optimization problem are raised. Because symmetries allow for faster and more accurate polynomial optimization, we have studied how an ODE's symmetries induce symmetries in the tangent dynamics of its Lyapunov vectors and, in turn, in the optimization problems that bound~LEs. 

The polynomial optimization problems that we formulate are SOS programs, and typically they must be solved numerically. If desired one could produce computer-assisted proofs by a similar approach, either by augmenting numerical computations with rational arithmetic~\cite{Peyrl2008} or interval arithmetic~\cite{Jansson2008}, or by using exact algebraic computation~\cite{Henrion2016a, Henrion2019} if its cost is not prohibitive. For polynomials of modest degree, the computations can sometimes guide analytical proofs also~\cite{Goluskin2018a}.

For illustration, we solved SOS programs numerically to find upper bounds on the maximal LE of two chaotic systems: the Lorenz system, which is dissipative, and the Henon--H\'eiles system, which is Hamiltonian. In each case the bounds become sharp as polynomial degrees are raised in the optimization problems, which also raises the computational cost. Bounds for the Lorenz system indicate that the maximal LE is attained on the fixed point at the origin, which has been proved analytically already~\cite{Leonov2016a}. Bounds for the H\'enon--Heiles system indicate that the maximal LE in a certain invariant set is attained on the shortest unstable periodic orbit.

The more typical approach to computing LEs, in contrast to ours, requires computing various approximate trajectories by numerical integration, along with the approximate LEs of these trajectories. Computing a variety of unstable trajectories can be hard on its own, but computing LEs on such orbits presents additional difficulty. On periodic orbits, imprecision in computation of the orbits can be magnified in the LE computation. On chaotic orbits, LEs are infinite-time averages approximated over finite times, but the convergence of these time averages is often very slow and without \emph{a priori} error estimates \cite[Chapter 1.9]{Sprott2010}. Such LE computations, to the extent that they can be trusted, are lower bounds on the maximal LE among all trajectories. Our approach provides upper bounds and thus complements numerical integration.

There are various natural extensions to the present work. For dynamics governed by discrete maps rather than differential equations, there is a direct analogue to the general variational formulation of \cref{prop: LE duality}, but additional steps are needed to arrive at an SOS relaxation similar to \cref{prop: LE duality SOS}. This is described in the doctoral thesis of the first author~\cite{Oeri2023}. Extensions to LEs governed by stochastic ODEs are immediate since the general method for bounding time averages that we have applied here has an analogue for stochastic ODEs~\cite{Fantuzzi2016}. In an example with linear deterministic dynamics subject to multiplicative noise, upper and lower bounds on the stochastic LE that essentially equal one another have been computed successfully~\cite{Kuntz2016}. We expect significantly more difficulty in cases where the deterministic dynamics is nonlinear and has different LEs on different trajectories. This is because computing bounds on stationary expectations for a stochastic system that would be violated by time averages in the corresponding deterministic system requires polynomials of high degree and leads to poor numerical conditioning~\cite{Fantuzzi2016}. Overcoming this obstacle would provide a powerful tool, for instance to find positive lower bounds on LEs of stochastic systems and thereby verify chaotic behavior.

\paragraph{Acknowledgements}
We thank Charles Doering, Giovanni Fantuzzi, Jeremy Parker, Sam Punshon-Smith, and Anthony Quas for some very helpful discussions in the course of this work. The discussions with Charles Doering were enabled by funding for his visit to the University of Victoria from the Pacific Institute for the Mathematical Sciences. Both authors were supported by the Canadian NSERC Discovery Grants Program via award numbers RGPIN-2018-04263, RGPAS-2018-522657, and DGECR-2018-0037.

\appendix

\section*{Appendix}
\label{sec: appendix}

\Cref{prop: sym V,prop: sym SOS} below concern symmetry in the general framework for bounding time averages that the present work has applied to bounding LEs. \Cref{prop: sym V} generalizes \cite[Proposition A.1]{Goluskin2019}, which was proved in the context of finite cyclic symmetry groups, to all compact subgroups of $GL(n)$. \Cref{prop: sym SOS} generalizes \cite[Theorem 2]{Lakshmi2020} in the same way. Compact subgroups of $GL(n)$ either are subgroups of $O(n)$ or are similarity transformations of such a subgroup as in \cref{eq: G similar} \cite[theorem 0.3.5]{Bredon1972}. The significance of \cref{prop: sym V,prop: sym SOS} is that symmetry can be imposed on the function $V$ being optimized in the minimization on the right-hand sides of \cref{eq: weak duality,eq: strong duality}, and on $V$ and the other tunable polynomials in the SOS relaxation on the right-hand sides of \cref{eq: weak duality 3,eq:constrained SOS program}, without changing the infima that give upper bounds on $\ov\Phi$. In the SOS programs, the polynomial expressions that are constrained to be SOS will be symmetric also, which can be exploited for faster and more accurate numerical solutions of SOS programs.

\begin{prop}
\label{prop: sym V}
Let $\cG$ be a compact subgroup of $GL(n)$. Suppose the set $\B \subset \Rm$ is $\cG$-invariant, the function $f:\cB\to\Rm$ is $\cG$-equivariant, and the function $\Phi:\cB\to\R$ is $\cG$-invariant. If there exist $V \in \cC^1(\cB)$ and $\U\in\R$ such that $\U - \Phi(x)-f(x)\cdot\nabla V(x) \geq 0$ for all $x\in\cB$, then there exists $\widehat{V} \in \cC^1(\cB)$ that is $\cG$-invariant and satisfies the same inequality.
\end{prop}

\begin{proof}
Because $\cG$ is compact it has a Haar probability measure $m$~\cite[proposition 11.4]{Folland1999} that is invariant in the sense that~\cite[theorem 3.1]{Bredon1972}
\beq
\label{eq: Haar integral}
\int_\cG \iota(\Lambda\Lambda')\,{\rm d} \mL = \int_\cG \iota(\Lambda)\,{\rm d} \mL
\eeq
for all $\Lambda'\in\cG$ and $\iota\in\cC(\cG)$. Define the symmetrization of $V$ as
\begin{align}
\label{eq: V symmetrization}
\widehat{V}(x) = \int_\cG V(\Lambda x ) \; {\rm d} \mL,
\end{align}
and note that $\widehat V$ is $\cG$-invariant by virtue of \cref{eq: Haar integral}. Since $\U - \Phi(x) - f(x)\cdot\nabla V(x) \geq 0$ holds for all $x\in\cB$, and $\cB$ is $\cG$-invariant, the same inequality holds when $x$ is replaced by $\Lambda x$ for any $\Lambda\in\cG$. Evaluating this inequality at $\Lambda x$ and then using the invariance of $\Phi$ and equivariance of $f$ under $\Lambda$ gives
\begin{align}
\label{eq: sym V proof 1}
\U - \Phi( x) - \Lambda f( x)\cdot\nabla V(\Lambda x)  \geq 0 \quad \forall x\in\cB.
\end{align}
The gradient in \cref{eq: sym V proof 1} is with respect to the entire argument $\Lambda x$, but this alternatively can be expressed using a gradient with respect to $x$ since
\beq
\label{eq: sym V proof 2}
\Lambda f(x)\cdot\nabla_{x'} V(x')|_{x'=\Lambda x} = f(x)\cdot\nabla_{x} V(\Lambda x).
\eeq
Substituting the right-hand side of \cref{eq: sym V proof 2} into \cref{eq: sym V proof 1} and integrating over $\cG$ against $m$ gives
\begin{align}
\U - \Phi(x) - f(x)\cdot \nabla \widehat{V}(x) \geq 0  \quad \forall x\in\cB,
\end{align}
which is the desired inequality for the symmetrization $\widehat V$.
\end{proof}

\begin{prop}
\label{prop: sym SOS}
Let $\cG$ be a compact subgroup of $GL(n)$. Suppose all $g_i,h_j\in\Rx$ that define $\cB$ via \cref{eq: semialgebraic set} are $\cG$-invariant, and suppose $\Phi\in\Rx$ and $f\in\R^n[x]$ are $\cG$-invariant and $\cG$-equivariant, respectively. If there exist $V,\sigma_i,\rho_j \in \Rx$ and $\U\in\R$ such that 
\beq
\label{eq: SOS pre-sym} 
B - \Phi - f\cdot \nabla V - \textstyle\sum_{i=1}^I\sigma_ig_i  - \sum_{j=1}^J\rho_jh_j  \in \Sigma_n \quad \text{ and }\quad \sigma_i \in \Sigma_n ~\text{ for }~ i=1,\ldots,I,
\eeq
then there exist $\cG$-invariant $\widehat{V},\widehat{\sigma}_i,\widehat{\rho}_j\in\Rx$ satisfying \cref{eq: SOS pre-sym} in place of $V,\sigma_i,\rho_j$. With these $\widehat{V},\widehat{\sigma}_i,\widehat{\rho}_j\in\Rx$, the expression in the first constraint of \cref{eq: SOS pre-sym} is $\cG$-invariant also.
\end{prop}

\begin{proof}
Define $\widehat{V}$ as the $\cG$-invariant symmetrization of $V$ using the invariant Haar probability measure $m$, as in \cref{eq: V symmetrization}, and define $\widehat{\sigma}_i$ and $\widehat{\rho}_j$ as the analogous symmetrizations of $\sigma_i$ and $\rho_j$. To see that $\sigma_i\in\Sigma_n$ implies $\widehat\sigma_i\in\Sigma_n$, let $b(x)$ be a vector of all monomials with degree up to half the degree of $\sigma_i$. There exists a symmetric matrix $Q_{\sigma_i}$ such that $\sigma_i(x)=b(x)^\T Q_{\sigma_i}b(x)$, and because $\sigma_i\in\Sigma_n$ it is possible to choose $Q_{\sigma_i}\succeq0$. For each $\Lambda \in \cG$, every entry in $b(\Lambda x)$ is a linear combination of the entries of $b(x)$, so there exists a matrix $\Gamma(\Lambda)$ such that $ b(\Lambda x) = \Gamma(\Lambda) b(x)$. Therefore, 
\beq
\label{eq: sigma_j symmetrized}
\widehat{\sigma}_i(x) = \int_\cG \sigma_i(\Lambda x)\, {\rm d}\mL
= b(x)^\T \underbrace{\left[ \int_\cG \Gamma(\Lambda)^\T Q_{\sigma_i}  \Gamma(\Lambda)\, {\rm d} \mL \right]}_{\widehat Q_{\sigma_i}} b(x).
\eeq
This formula makes clear that $\widehat{\sigma}_i$ is a polynomial, and it is also SOS since $\Gamma(\Lambda)^\T Q_{\sigma_i} \Gamma(\Lambda)\succeq0$ for each $\Lambda \in \cG$ implies $\widehat Q_{\sigma_i}\succeq0$. Similar arguments confirm that $\widehat V$ and $\widehat \rho_j$ are polynomials but not that they are SOS since $V$ and $\rho_j$ need not be.

The argument for the first SOS constraint is similar to the argument for $\widehat\sigma_i$. There exist a (possibly different) basis vector $b(x)$ and a matrix $Q\succeq0$ such that
\beq
B - \Phi - f\cdot \nabla V - \textstyle\sum_{i=1}^I\sigma_ig_i - \sum_{j=1}^J\rho_jh_j = b^\T Q\,b
\eeq
since the left-hand expression is assumed to be SOS. The above equality holds for all $x\in\R^n$, including with $x$ replaced by $\Lambda x$. Substituting $\Lambda x$ for $x$, we use the $\Lambda$-invariance of $V,g_i,h_j$ and the $\Lambda$-equivariance of $f$, along with \cref{eq: sym V proof 2}, to find
\begin{multline}
B - \Phi(x) - f(x)\cdot \nabla_xV(\Lambda x) - \textstyle\sum_{i=1}^J\sigma_i(\Lambda x)g_i(x) - \textstyle\sum_{j=1}^I\rho_j(\Lambda x)h_j(x) \\ = b(x)^\T\Gamma(\Lambda)^\T Q\,\Gamma(\Lambda)b(x).
\end{multline}
As in \cref{eq: sigma_j symmetrized}, integrating both sides of the above equality over $\cG$ against $m$ gives
\beq
\label{eq: symmetrized SOS}
B - \Phi - f\cdot \nabla \widehat V - \textstyle\sum_{i=1}^J\widehat\sigma_ig_i - \textstyle\sum_{j=1}^I\widehat\rho_jh_j=
b^\T \underbrace{\left[ \int_\cG \Gamma(\Lambda)^\T Q \Gamma(\Lambda)\; d \mL \right]}_{\widehat Q} b.
\eeq
The integrand defining $\widehat Q$ is positive semidefinite for each $\Lambda$, so $\widehat Q\succeq0$ also, thus the left-hand expression is SOS. This completes the proof that $\widehat{V},\widehat{\sigma}_i,\widehat{\rho}_j$ satisfy \cref{eq: SOS pre-sym}. The left-hand expression in \cref{eq: symmetrized SOS} is also $\cG$-invariant since it can be written as an integral against the Haar probability measure of $\cG$, and this confirms the second part of the claim.
\end{proof}

\bibliography{main.bbl}

\end{document}